\documentclass[12pt]{article}

\usepackage[top=50mm, bottom=50mm, left=50mm, right=50mm]{geometry}
\usepackage{lineno}
\usepackage{amssymb}
\usepackage{amsmath}
\usepackage{amsthm}
\usepackage{epsfig}
\usepackage{graphicx}
\usepackage{graphics}
\usepackage{float}
\usepackage{subfigure}
\usepackage{multirow}
\usepackage{color}
\usepackage{fullpage}
\usepackage[normalem]{ulem} 
\usepackage{makeidx}
\usepackage{xspace}
\usepackage{wrapfig}
\usepackage{todonotes}
\usepackage{hyperref}
\usepackage{cleveref}
\usepackage{cite}
\makeindex

\newtheorem{Definition}{Definition}[section]
\newtheorem{theorem}[Definition]{Theorem}
\newtheorem{Corollary}[Definition]{Corollary}
\newtheorem{Theorem}[Definition]{Theorem}
\newtheorem{Lemma}[Definition]{Lemma}

\newtheorem{Proposition}[Definition]{Proposition}

\newtheorem{Remark}[Definition]{Remark}

\definecolor{darkred}{rgb}{1, 0.1, 0.3}
\definecolor{darkblue}{rgb}{0.1, 0.1, 1}
\definecolor{darkgreen}{rgb}{0,0.6,0.5}

\newcommand {\mm}[1] {\ifmmode{#1}\else{\mbox{\(#1\)}}\fi}

\newcommand{\C}{\mathbb{C}} % komplexe Zahlen
\newcommand{\T}{\mathbb{T}} % Körper
\newcommand{\R}{\mathbb{R}} % reelle Zahlen
\newcommand{\Q}{\mathbb{Q}} % rationale Zahlen
\newcommand{\Z}{\mathbb{Z}} % ganze Zahlen
\newcommand{\HH}{\mathcal{H}}
\newcommand{\MM}{\mathcal{M}}

\newcommand{\LL}{\mathcal{L}}
\newcommand{\VV}{\mathcal{V}}
\newcommand{\dd}{\mathrm{d}}

\newcommand{\diff}{\frac{\mathrm{d}}{\mathrm{d}t}\Big|_{t=0}}

\begin{document}

\title{On the Hofer--Zehnder capacity for twisted tangent bundles over closed surfaces}
 
\author{
{Johanna Bimmermann}
}

%\institute{Stockholm University, Stockholm, Sweden\\\email{elena.touli@math.su.se}\andthe Ohio State University Columbus, Ohio, U.S.A.\\\email{yusu@cse.ohio-state.edu}\\}\\

%\authorrunning{Mokhov, Sutcliffe and Voronkov}

% \title{FPT-Algorithms for computing Gromov-Hausdorff and interleaving distances between trees}
% \author{Elena Farahbkhsh Touli} \and \author{y}
%\date{}

%\setcounter{page}{0}
\maketitle

\begin{abstract}
\noindent
We determine the Hofer--Zehnder capacity for twisted tangent bundles over closed surfaces for (i) arbitrary constant magnetic fields on the two-sphere and (ii) strong constant magnetic fields for higher genus surfaces. On $S^2$ we further give an explicit $\text{SO}(3)$-equivariant compactification of the twisted tangent bundle to $S^2\times S^2$ with split symplectic form. The former is the phase space of a charged particle moving on the two-sphere in a constant magnetic field, the latter is the configuration space of two massless coupled angular momenta.
\end{abstract}

\section{Introduction and main results}
The notion of symplectic capacities was developed to investigate the existence of symplectic embeddings. As symplectomorphisms are always volume preserving one could ask whether a symplectic embedding $M\hookrightarrow N$ exists if and only if $\textrm{Vol}(M)\leq\text{Vol}(N)$. The answer is no in dimension larger than two and was given by M. Gromov in 1985 with his non-squeezing theorem \cite{Gr85}. This means that there must be more global symplectic invariants than volume. A class of such invariants is given by symplectic capacities as introduced by H. Hofer and E. Zehnder in \cite{HZ94}. There, they constructed a special capacity, now known as the Hofer--Zehnder capacity, relating embedding problems with the dynamics on symplectic manifolds. Very importantly, its finiteness implies the existence of periodic orbits on almost all compact regular energy levels (\cite[Ch. 4]{HZ94}).
\begin{Definition}
Let $(M,\omega)$ be a symplectic manifold. We call a smooth Hamiltonian function $H:M\to\R$ admissible if there exists a compact subset $K\subset M\setminus \partial M$ and a non-empty open subset $U\subset K$ such that
\begin{itemize}
    \item[a)] 
    $
    H\vert_{M\setminus K}=\max H\ \ \text{and}\ \ H\vert_{U}=0,
    $
    \item[b)] $0\leq H(x)\leq \max H$ for all $x\in M$.
\end{itemize}
Denote by $\mathcal{H}(M)$ the set of admissible functions and by $\mathcal{P}_{\leq 1}(H)$ the set of non-constant periodic solutions to the Hamiltonian equations with period at most one. The Hofer--Zehnder capacity of a symplectic manifold $(M,\omega)$ is then defined as
$$
c_{\mathrm{HZ}}(M,\omega):=\sup\lbrace \max H\ \vert\ H\in\HH(M), \mathcal{P}_{\leq 1}(H)=\emptyset \rbrace.
$$
Further one can look at this capacity with respect to a fixed free homotopy class of loops $\nu$. We denote
$$
\mathcal{P}_T(H;\nu):=\lbrace \gamma\in C^\infty(\R/T\Z, M)\ \vert \ \dot\gamma(t)=X_H(\gamma(t))\neq 0;\ [\gamma]=\nu\rbrace
$$
the set of non-constant $T$-periodic solutions to the Hamiltonian equations in the class $\nu$ and by $\mathcal{P}_{\leq T}(M,\nu)$ the set of non-constant periodic solutions in class $\nu$ with period less or equal to $T$. The Hofer--Zehnder capacity with respect to this free homotopy class is defined to be
$$
c_{\mathrm{HZ}}^\nu(M,\omega)=\sup\lbrace \max H\ \vert\ H\in \HH(M), \mathcal{P}_{\leq 1}(H;\nu)=\emptyset\rbrace.
$$
\end{Definition}
\noindent
Loosely speaking the Hofer--Zehnder capacity tells us, how much a Hamiltonian function can oscillate before fast periodic solutions (namely with period at most one) appear.\\
As most capacities, the Hofer--Zehnder capacity is in general hard to compute and not known for many symplectic manifolds. In this paper we determine it for certain domains in twisted tangent bundles of closed surfaces.\\
The tangent bundle $T\Sigma$ is called twisted if we add to the canonical symplectic form $\dd\alpha$ a magnetic term, i.e. the pullback of a 2-form $\varpi\in\Omega^2(\Sigma)$ to the tangent bundle. On a surface choosing a metric determines a unique function $f:\Sigma\to\R$ such that $\varpi=f\cdot\sigma$, where $\sigma$ denotes the area form induced by the metric. The function $f$ can be interpreted as the strength of the magnetic field and we will restrict to constant fields, i.e. $f(x)=s\in \R$ for all $x\in \Sigma$. The main theorem of this paper gives the value of the Hofer-Zehnder capacity of disc-subbundles of the (constantly) twisted tangent bundle $(T\Sigma,\omega_s:=\dd\alpha-s\pi^*\sigma)$.
\begin{theorem}\label{thm1}
Let $(\Sigma,g_\kappa)$ be a closed connected orientable Riemannian surface with constant curvature $\kappa$. Denote by $\sigma_\kappa$ the corresponding area form, by $\alpha$ the canonical one-form on $T\Sigma$ and define the disc bundle $D_\lambda\Sigma:=\lbrace (x,v)\in T\Sigma\ \vert\ \vert v\vert < \lambda\rbrace$ of radius $\lambda$ with respect to $g_\kappa$. Then, whenever $s^2+\kappa\lambda^2>0$ for some $s\in\R$, we have
$$
c^0_{\mathrm{HZ}}(D_\lambda\Sigma, \dd\alpha-s\pi^*\sigma_\kappa)=\left\{\begin{array}{ll} \frac{2\pi}{\kappa}\left( \sqrt{s^2+\kappa\lambda^2}-\vert s\vert\right), & \text{for}\ \kappa\neq 0, \\
         \frac{\pi\lambda^2}{\vert s\vert}, & \text{for}\ \kappa=0.\end{array}\right.
$$
\end{theorem}
\noindent
The theorem covers three types of surfaces: Spheres ($\kappa>0$), the torus ($\kappa=0$) and higher genus surfaces $(\kappa< 0)$.
The assumption $s^2+\kappa\lambda^2>0$ does not put any additional constraint on the sphere, for the torus it tells us that the magnetic field does not vanish, i.e. $s\neq 0$, and for higher genus surfaces it tells us to look at strong magnetic fields, i.e. $\vert s\vert >\sqrt{-\kappa}\lambda$.
\begin{Remark}
While $c_{\mathrm{HZ}}\leq c^0_{\mathrm{HZ}}$ always holds, we will in these three cases find that actually $$c^0_{\mathrm{HZ}}(D_\lambda\Sigma, \dd\alpha-s\pi^*\sigma_\kappa)=c_{\mathrm{HZ}}(D_\lambda\Sigma, \dd\alpha-s\pi^*\sigma_\kappa).$$
\end{Remark}
\begin{Remark}
The result extends continuously to the limit $s^2+\kappa\lambda^2=0$. This happens either when looking at the torus with vanishing magnetic field ($\kappa=0, s=0$) or on higher genus surfaces whenever $s=\sqrt{\vert \kappa\vert}\lambda$. In these two cases we find
$$
c^0_{\mathrm{HZ}}(D_\lambda\Sigma, \dd\alpha-s\pi^*\sigma_\kappa)=\left\{\begin{array}{ll} \frac{2\pi\vert s\vert}{\vert \kappa\vert}, & \text{for}\ \kappa\neq 0, \\
         \infty, & \text{for}\ \kappa=0.\end{array}\right.
$$
In the limit case for higher genus surfaces it follows as the Hofer--Zehnder capacity must per definition be lower semi-continuous in $\lambda$. On the torus without magnetic field we observe that the kinetic Hamiltonian has no contractible orbits. This immediately yields $c_\mathrm{HZ}^0(D_\lambda\T^2,\dd\alpha)=\infty$.
\end{Remark}
\noindent
For these three cases the Hofer--Zehnder capacity can be computed in the same manner. The rough idea for finding a lower bound of the Hofer-Zehnder capacity is that the periodic solutions to the kinetic Hamiltonian $E(x,v)=\frac{1}{2}g_x(v,v)$ are contractible or even more specificly geodesic circles. It is fairly easy to modify the kinetic Hamiltonian into an admissible Hamiltonian (of the form $f(E)$) that does not admit fast periodic solutions and yields a lower bound for the Hofer--Zehnder capacity. Finding an upper bound is somewhat more involved and consists of three steps.\\
\textit{1. Symplectization of fibers:} Construction of a symplectomorphism
$$
F: (D_0^\lambda \Sigma, \omega_s)\to (D^b_a \Sigma, \Tilde{\omega}),
$$
where $\Tilde{\omega}$ is such that the fibers are symplectic, $D^b_a\Sigma=\lbrace (x,v)\in T\Sigma\ \vert\ a<\vert v\vert < b\rbrace$ and $a,b$ are some constants depending on $\lambda$ and $s$. \\
\textit{2. Compactification:} We will compactify this bundle fiberwise into a (topologically trivial) sphere bundle.\footnote{This compactification can be seen as concrete case of the symplectic cut construction by Lerman \cite{ler}.  } This yields a symplectic embedding $(D^\lambda_0 \Sigma,\omega_s)\hookrightarrow (S^2\times \Sigma,\tilde\omega)$ which can be extended to the zero section.\\
\textit{3. Application of Lu's theorem:} We can then use a theorem by G. Lu \cite[Thm. 1.10]{Lu06} to show that the symplectic area of a fiber yields an upper bound to the Hofer--Zehnder capacity of the twisted disc bundle.\\
Finally we will see that upper and lower bound agree and therefore determine the Hofer--Zehnder capacity.\\
%\todo[inline]{Mention of symplectic cuts from Lerman, Albers, Geiges, Zehmisch.}
In the case of $D_\lambda S^2$ it is possible to do symplectization and compactification more explicitly and respecting the symmetry group $\text{SO}(3)$.  
\begin{theorem}\label{thm4}
There is an $\mathrm{SO}(3)$-equivariant symplectomorphism
$$
F: (D_\lambda S^2,\omega_s)\to (S^2\times S^2\setminus \Delta, R_1\sigma\oplus R_2\sigma),
$$
where $\Delta\subset\ S^2\times S^2$ is the diagonal and $R_1,\ R_2 >0$ satisfy
$$
s=R_2-R_1,\ \ \lambda=2\sqrt{R_1R_2}.
$$
\end{theorem}
\begin{Remark}
As suggested by Richard Hind $(S^2\times S^2,\sigma\oplus\sigma)$ can be identified with the quadric $x^2+y^2+z^2=w^2$ in $\C \mathbb P^3$ with Fubini--Study symplectic form and the disc bundle $(DS^2,\dd\alpha)$ is included as the affine quadric $x^2+y^2+z^2=1$.
Thus, the existence of such a symplectomorphism in the case of vanishing magnetic field (s=0) has been known before (Entov--Polterovich and Zapolsky, unpublished). 
\end{Remark}
\noindent
The target space of this symplectomorphism $(S^2\times S^2, R_1\sigma\oplus R_2\sigma)$ admits a physical interpretation. It is the configuration space of two coupled massless oscillators $N,S\in S^2\subset\R^3$. The Hamiltonian of the potential energy of this system is given by $H=R_1R_2(1-N\cdot S)$ where $\cdot$ denotes the standard scalar product in $\R^3$ and $R_1R_2$ determines the strength of the coupling (spring constant). We will find that $F$ relates $H$ to the kinetic Hamiltonian $E$ via
$$
H\circ F=\frac{1}{2}\lambda^2-E.
$$
\begin{Remark}
Finiteness of the Hofer--Zehnder capacity was already known due to Benedetti and Zehmisch in the spherical case \cite{BZ15}. Further, Macarini \cite{Mac03} shows that for strong fields on higher genus surfaces the Hofer--Zehnder capacity of $(D_0^\lambda\Sigma,\omega_s)$ is finite. We extend this result to $D_\lambda\Sigma$ and compute the actual value.  \\
The Hofer--Zehnder capacity for the twisted disc bundle of the torus was already known due to V. Ginzburg \cite[Ch. 5]{Gb96}, who gives a symplectomorphism between $(T^*\T,\omega_s)$ and $(\R^2\times \T, \Omega\oplus\sigma_g)$ where $\Omega$ is the standard area form of $\R^2$, and A. Floer, H. Hofer and C. Viterbo  \cite{FHV89} who computed $c_{\mathrm{HZ}}(D_\lambda\times\T,\Omega\oplus\sigma_g)$ (see also \cite[Ch. 3.5, Thm. 6]{HZ94}). We included it for completeness and since it can be computed using the same technique as for the other cases.
\end{Remark}
\begin{Remark}\label{rem}
It is unknown (at least to the author) what the Hofer--Zehnder capacity for the disc bundle of the torus with vanishing magnetic field and for higher genus surfaces with weak magnetic fields is. In the case of the torus finiteness follows from \cite[Prop.4, Ch.4]{HZ94}, but for higher genus surfaces even finiteness remains unclear.
\end{Remark}
\noindent
In the cases of Remark \ref{rem}, one can look at a relative version of the Hofer--Zehnder capacity (defined by V. Ginzburg and B. Gürel in \cite{VG2003}).

\begin{Definition}
For a subset $Z\subset M$ that doesn't touch the boundary, i.e. $\mathrm{cl}(Z)\cap\partial M=\emptyset $, we denote by $\HH(M,Z)$ the set of smooth functions satisfying 
\begin{itemize}
    \item[a)] 
    $
    H\vert_{M\setminus K}=\max H\ \ \text{and}\ \ H\vert_{U}=0,
    $
    \item[b)] $0\leq H(x)\leq \max H$ for all $x\in M$,
\end{itemize}
for an open neighborhood $U\supset Z$ and a compact set $K\supset U$. The relative Hofer--Zehnder capacity is then defined as 
$$
c_{\mathrm{HZ}}(M, Z, \omega):=\sup\lbrace \max H\ \vert\ H\in\HH(M,Z), \mathcal{P}_{\leq 1}(H)=\emptyset\rbrace.
$$
\end{Definition}
\begin{Remark}
Observe that clearly $c_{\mathrm{HZ}}(M,Z,\omega)\leq c_{\mathrm{HZ}}(M,\omega)$ for any $Z\subset M$.
\end{Remark}
\begin{Remark}
As for the Hofer--Zehnder capacity there is an almost existence result in case of finite relative Hofer--Zehnder capacity \cite[Thm. 2.14]{VG2003}. It says that if $c_{\mathrm{HZ}}(M,Z,\omega)<\infty$ and $H: M\to \R$ is a proper smooth function with $H\vert_Z=\min H$, then almost all compact regular energy levels carry periodic orbits.
\end{Remark}
\noindent
We further denote by
$$
l_\nu=\inf\lbrace \text{length}(\gamma)\ \vert\ \gamma\ \text{is a closed geodesic of}\ (\Sigma,g)\ \text{in the class}\ \nu \rbrace.
$$
J. Weber determined in \cite[Thm. 4.3]{Wbr06} the value of the BPS-capacity of unit disc-bundles with canonical symplectic structure relative to the zero section to be
\begin{equation}\label{rell}
    c_{BPS}^\nu(D_1\Sigma,\Sigma, \omega_0)=l_\nu.
\end{equation}
\noindent
Ginzburg and Gürel showed in \cite{VG2003} section 2.2 that the relative Hofer--Zehnder capacity coincides with the BPS-capacity for standard bundles as defined in \cite{Wbr06}.
The theorem by Weber therefore covers the case of the torus with no magnetic field.
We further used this theorem to compute $c_{\mathrm{HZ}}^\nu(D_\lambda \Sigma, D_{ \frac{\vert s\vert}{\sqrt{-\kappa}}}\Sigma, \omega_s)$ for $\Sigma$ a higher genus surface and $s^2+\kappa\lambda^2<0$. Even though the value was already given in Ginzburg \cite[Ex. 4.7] {Gb04}, we wrote down a detailed proof as we could not find the details in the literature.

\begin{theorem}\label{thm3}
If $s^2+\kappa\lambda^2<0$ the Hofer--Zehnder capacity of $(D_\lambda\Sigma,\omega_s)$ relative to $D_{ \frac{\vert s\vert}{\sqrt{-\kappa}}}\Sigma$ is given by
$$
c_{\mathrm{HZ}}^\nu\left(D_\lambda \Sigma, D_{\frac{\vert s\vert}{\sqrt{-\kappa}}}\Sigma, \omega_s\right)=l_\nu\sqrt{-(\kappa\lambda^2+s^2)},
$$
where $l_\nu$ denotes the shortest length of a closed geodesic in the free-homotopy class $\nu$ of loops in $\Sigma$. 
\end{theorem}

\noindent
\textbf{Structure of the paper.}\\
In Section \ref{sec2} we introduce twisted tangent bundles, a suitable frame of $T(D_0^\lambda\Sigma)$ and explain some basics about magnetic systems on surfaces. Section \ref{sec:3} explains how we obtain an upper bound for the Hofer--Zehnder capacity with the help of a theorem by G. Lu. We give the definition of a pseudo-capacity in Section \ref{sec3.2} and explain how it is related to the Hofer--Zehnder capacity. In Section \ref{sec3.3} we state the theorem by G. Lu that gives an upper bound for the pseudo-capacity in terms of the symplectic area of a homology class for which a Gromov--Witten invariant does not vanish.
This Gromov--Witten invariant is determined in Section \ref{sec3.1}. We continue with proving Theorem \ref{thm1} in section \ref{sec4} and Theorem \ref{thm3} in Section \ref{sec5}. The equivariant symplectomorphism of Theorem $\ref{thm4}$ is then built in Section \ref{sec6}.\\

\noindent
\textbf{Acknowledgment.}\\
I want to thank JProf.\ Gabriele Benedetti for suggesting the topic of this paper and his unlimited support and help guiding me through the process of writing. I also want to thank Valerio Assenza and Maximilian Schmahl for the helpful discussions and their devoted proofreading.\\
The author further acknowledges funding by the Deutsche Forschungsgemeinschaft (DFG, German Research Foundation) – 281869850 (RTG 2229). 
\section{Twisted tangent bundles}\label{sec2}
In this section we quickly introduce all notations and tools on twisted tangent bundles we need. More details can be found in \cite{Bd14}. Let $\Sigma$ be a smooth connected orientable closed surface, let $g=g_\kappa$ be a Riemannian metric of constant curvature $\kappa$ on $\Sigma$. We will study the tangent (and cotangent) bundle of this surface. Denote by $\Pi:T^*\Sigma\to \Sigma$ the canonical projection. The metric defines an isomorphism
\begin{equation}\label{eq:4}
    T_x\Sigma\to T_x^*\Sigma;\qquad v\mapsto p:=g_x(v,\ \cdot\ ).
\end{equation}

\noindent
We can pointwise define the Liouville 1-form on $T^*\Sigma$
$$
\alpha_{(x,p)}:=p\circ \dd\Pi_{(x,p)}:\ T_{(x,p)}T^*\Sigma\to \R,
$$
which can be pulled back via the metric isomorphism \eqref{eq:4} to a canonical 1-form on $T\Sigma$ also denoted by $\alpha$
$$
\alpha_{(x,v)}:\ T_{(x,v)}T\Sigma\to \R;\qquad \xi\mapsto g_x(v,\dd\pi_{(x,v)}\xi).
$$
Here, $\pi:T\Sigma\to \Sigma$ is the canonical projection. The exterior derivative $\omega_0:=\dd\alpha$ defines a canonical symplectic form on either $T\Sigma$ or $T^*\Sigma$.
Furthermore we denote by  $\sigma\in\Omega^2(\Sigma)$ the Riemannian area form with respect to a given orientation on $\Sigma$. We will study the twisted tangent bundle
$$
(T\Sigma,\omega_s:=\dd\alpha-s\pi^*\sigma)
$$
for some fixed $s\in\R$.
We restrict our attention to disc-bundles
$$
D_\lambda \Sigma:=\lbrace (x,v)\ \vert g_x(v,v) < \lambda^2\rbrace ,\qquad \lambda\geq 0.
$$
We can further reduce to the case $s\geq 0$ as changing the orientation of $\Sigma$ corresponds to exchanging $s$ and $-s$ and everything we do works for an arbitrary orientation of $\Sigma$.\\
\noindent
One defines a complex structure $\iota$ on $\Sigma$ via the relation
$$
\sigma_x(\ \cdot\ ,\ \cdot\ )= g_x(\iota_x\cdot,\ \cdot\ ).
$$
\noindent
Denote by $T^0\Sigma$ the tangent bundle without the zero section. We will now define a frame of $T(T^0\Sigma)$. Denote by 
$$\mathcal{L}^\mathcal{V}_{(x,v)}: T_x\Sigma\to T_{(x,v)}T\Sigma
$$
the vertical lift associated to the canonical projection $\pi:T\Sigma\to \Sigma$ and by
$$\mathcal{L}^\mathcal{H}_{(x,v)}: T_x\Sigma\to T_{(x,v)}T\Sigma
$$
the horizontal lift induced by the Levi-Civita connection $\nabla$ of $(\Sigma,g)$. 
A frame of $T(T^0\Sigma)$ is given by
\begin{align*}
    Y_{(x,v)}&:=\LL^{\VV}_{(x,v)}(v),\\
V_{(x,v)}&:=\LL^\VV_{(x,v)}(\iota_xv),\\
X_{(x,v)}&:=\LL^\HH_{(x,v)}(v),\\
H_{(x,v)}&:=\LL^\HH_{(x,v)}(\iota_x v).
\end{align*}
These vector fields satisfy the commutator relations
\begin{align*}
    [Y,X]&=X,\ \ \ \ \ [Y,H]=H,\ \ \ \ \  [Y,V]=0,\ \ \ \\
    [V,X]&=H,\ \ \ \ \ [V,H]=-X,\ \ \ \ \  [X,H]=2E\kappa V,\ \ \ 
\end{align*}
where $E: TM\to \R; (x,v)\mapsto \frac{1}{2}g_x(v,v)$ is the kinetic energy. The dual co-frame of $(Y,V,X,H)$ is given by 
$$
\left ( \frac{\dd E}{2E},\tau,\frac{\alpha}{2E},\frac{\eta}{2E}\right),
$$
where $\tau$ is the angular form associated to the Levi-Civita connection and $\eta:=\iota^*\alpha$. One can use the formula
$$
\dd\xi(U_1,U_2)=U_1(\xi(U_2))-U_2(\xi(U_1))-\xi([U_1,U_2])
$$ for an arbitrary one-form $\xi$ and vector fields $U_1,U_2$, to calculate exterior derivatives and finds
\begin{align}\label{rel}
    \dd\tau&=-\frac{\kappa}{2E}\alpha\wedge\eta=-\kappa\pi^*\sigma,\\
    \dd\alpha&=\frac{1}{2E}\dd E\wedge\alpha+\tau\wedge\eta,\\
    \dd\eta&=\frac{1}{2E}\dd E\wedge \eta+\alpha\wedge\tau,
\end{align}
where $\sigma$ is the area form on $\Sigma$.
\noindent
Using these relations we can prove the following proposition that will be needed in the search for Hamiltonian orbits.
\begin{Proposition}
The Hamiltonian vector field $X_E$ corresponding to the kinetic Hamiltonian $E(x,v)=\frac{1}{2}\vert v\vert^2$ is given by
$$
X_E=X+sV.
$$
\end{Proposition}
\begin{proof}
The vector field $X_E$ is defined via 
$$
-\dd E=\omega_s(X_E,\ \cdot\ ).
$$
This translates to
\begin{align*}
    -\dd E &= \dd \alpha(X_E,\ \cdot\ )-s\pi^*\sigma(X_E,\ \cdot\ )\\
    &=\frac{1}{2E}\dd E\wedge\alpha (X_E,\ \cdot\ )+\tau\wedge\eta(X_E,\ \cdot\ )-\frac{s}{2E}\alpha\wedge\eta(X_E,\ \cdot\ )\\
    &= -\frac{\alpha(X_E)}{2E}\dd E+\frac{\dd E(X_E)}{2E}\alpha+\tau(X_E)\eta-  \eta(X_E)\tau+\frac{s\eta(X_E)}{2E}\alpha-\frac{s\alpha(X_E)}{2E}\eta,
\end{align*}
and we read off
\begin{align*}
    1&=\frac{\alpha(X_E)}{2E}\\
    0&=-\tau(X_E)+s\frac{\alpha(X_E)}{2E}\\
    0&=-\frac{\dd E(X_E)}{2E}-s\frac{\eta(X_E)}{2E}\\
    0&=\eta(X_E)
\end{align*}
which is uniquely solved by $X_E=X+sV$.
\end{proof}

\noindent
In the proofs of Theorem \ref{thm1} and Theorem \ref{thm3} we will further need the following two lemmas from \cite[Thm.\ A.1.]{BR19}, that are again proven using the relations (\ref{rel})-(5).
\begin{Lemma}\label{lem4}
Denote by $\Phi_b$ the flow of $-H$ for time $b\in \R$, then
\begin{align*}
   \Phi_b^*\tau&=-\frac{\sqrt{\kappa}}{\sqrt{2E}}\sin(\sqrt{2E\kappa}b)\alpha+\cos(\sqrt{2E\kappa}b)\tau, \\
   \Phi_b^*\alpha&=\cos(\sqrt{2E\kappa}b)\alpha+\frac{\sqrt{2E}}{\sqrt{\kappa}}\sin(\sqrt{2E\kappa}b) \tau,
\end{align*}
where we keep the relations $\sin(ia)=i\sinh(a)$ and $\cos(ia)=\cosh(a)$ in mind.
\end{Lemma}
\begin{proof}
Let $\xi$ be some one-form. We make the ansatz $\Phi^*_b\xi=x\alpha+y\eta+z\tau+w\dd E$, then
\begin{align*}
    0&=\frac{\dd}{\dd b}\Phi^*_{-b}\Phi^*_b\lambda=\frac{\dd}{\dd b}\Phi^*_{-b}(x\alpha+y\eta+z\tau+w\dd E)\\
    &= \dot x\alpha+\dot y\eta+\dot z\tau+\dot w\dd E+x\LL_H\alpha+y\LL_H\eta+z\LL_H\tau+w\LL_H\dd E\\
    &=\dot x\alpha+\dot y\eta+\dot z\tau+\dot w\dd E+x\iota_H\dd\alpha+y\iota_H\dd\eta+y\dd(\iota_H\eta)+z\iota_H\dd\tau\\
    &=\dot x\alpha+\dot y\eta+\dot z\tau+\dot w\dd E-2Ex\tau-y\dd E+2y\dd E+\kappa z\alpha
\end{align*}
where we applied the relations (\ref{rel})-(5).
Thus we find a system of ordinary differential equations
$$
\dot x=-\kappa z,\ \ \dot y=0,\ \ \dot z=2Ex,\ \ \dot w=-y,
$$
which can be solved uniquely by fixing initial conditions. In the case $\xi=\tau$ these are 
$$x(0)=0,\ \ y(0)=0,\ \, z(0)=1,\ \, w(0)=0.$$
Then the solution is
$$
x(b)=- \frac{\sqrt{\kappa}}{\sqrt{2E}}\sin(\sqrt{2E\kappa}b),\ \  y(b)=0,\ \ z(b)=\cos(\sqrt{2E\kappa}b),\ \ w(b)=0.
$$
Similarly, we can choose the initial conditions such that $\xi=\alpha$. Then
$$x(0)=1,\ \ y(0)=0,\ \, z(0)=0,\ \, w(0)=0$$
and the solution is
$$
x(b)=\cos(\sqrt{2E\kappa}b),\ \  y(b)=0,\ \ z(b)= \frac{\sqrt{2E}}{\sqrt{\kappa}}\sin(\sqrt{2E\kappa}b),\ \ w(b)=0.\ \qedhere
$$
\end{proof}
\noindent
\begin{Definition}\label{def}
For a real number $a\in\R$ denote by
$$
m_a: T\Sigma\to T\Sigma;\qquad (x,v)\mapsto (x,av)
$$
the flow of $Y$ for time $\ln(a)$.
\end{Definition}
\noindent
Analogouly to the proof of Lemma \ref{lem4} one shows the following lemma.
\begin{Lemma}\label{lem3}
The scaling of fibers $m_a$ satisfies
$$
    m^*_a\tau=\tau,\qquad
    m^*_a\alpha=a\alpha,\qquad
    m^*_a\eta=a\eta.\qquad\qedsymbol
$$
\end{Lemma}

\section{Hofer--Zehnder capacity via Gromov--Witten invariants}\label{sec:3}
As discovered by Hofer--Viterbo \cite{HV92} and Liu--Tian \cite{LT00}, there is a really interesting connection between the existence of closed trajectories and Gromov--Witten invariants. G. Lu \cite{Lu06} uses their results to show that under certain assumptions a non-vanishing Gromov--Witten invariant for a homology class $[A]\in H_2(M)$ implies that the symplectic area $\omega([A])$ is an upper bound for the Hofer--Zehnder capacity. In this chapter we shall first introduce a pseudo symplectic capacity, then quickly introduce pseudoholomorphic curves and Gromov--Witten invariants and explain Lu's theorem \cite[Thm.\ 1.10]{Lu06} in the trivial sphere bundle $S^2\times\Sigma$, which is the set up we are interested in.

\subsection{Pseudo symplectic capacity of Hofer--Zehnder type}\label{sec3.2}
G. Lu \cite{Lu06} relates Gromov--Witten invariants and the Hofer--Zehnder capacity using the more general notion of pseudo-capacities.

\begin{Definition}\label{def:3}\ \\
For a connected symplectic manifold $(M,\omega)$ of dimension at least four and two nonzero homology classes $a_-,a_+ \in H_*(M)$, we call a smooth function $H : M \to \R$, $(a_-,a_+)$-admissible if there exist two compact submanifolds $P$ and $Q$ of $M$ with connected smooth boundaries and of codimension zero such that the following conditions hold:
\begin{enumerate}
    \item  $P \subset \mathrm{Int}(Q)$ and $Q \subset \mathrm{Int}(M)$ ; 
    \item $H\vert_P = 0$ and $H\vert_{M\setminus \mathrm{Int}(Q)} = \max H$;
    \item $0 \leq H \leq \max H$;
    \item There exist cycle representatives of $a_-$ and $a_+$, still denoted by $a_-,a_+$, such that $\mathrm{supp}(a_-) \subset \mathrm{Int}(P)$ and $\mathrm{supp}(a_+) \subset M\setminus Q$;
    \item There are no critical values in $(0,\varepsilon) \cup (\max H - \varepsilon,\max H)$ for a small $\varepsilon = \varepsilon(H) > 0$.
\end{enumerate}
We denote by $\mathcal{H}_{a}(M;a_-,a_+)$ the set of all $(a_-,a_+)$-admissible functions. We call 
$$
C^{(2)}_{\mathrm{HZ}}(M,\omega;a_-,a_+) := \sup\lbrace\max H\ \vert\ H\in \HH_{a}(M;a_-,a_+), \mathcal{P}_{\leq 1}(H)=\emptyset\rbrace
$$
pseudo symplectic capacity of Hofer--Zehnder type. We further denote by
$$
C^{(2o)}_{\mathrm{HZ}}(M,\omega;a_-,a_+) := \sup\lbrace\max H\ \vert\ H\in \HH_{a}(M;a_-,a_+), \mathcal{P}_{\leq 1}(H;0)=\emptyset\rbrace
$$
the pseudo symplectic capacity of Hofer--Zehnder type with respect to the homotopy class of contractible loops.
\begin{figure}
    \centering
    \includegraphics[width=0.6\textwidth]{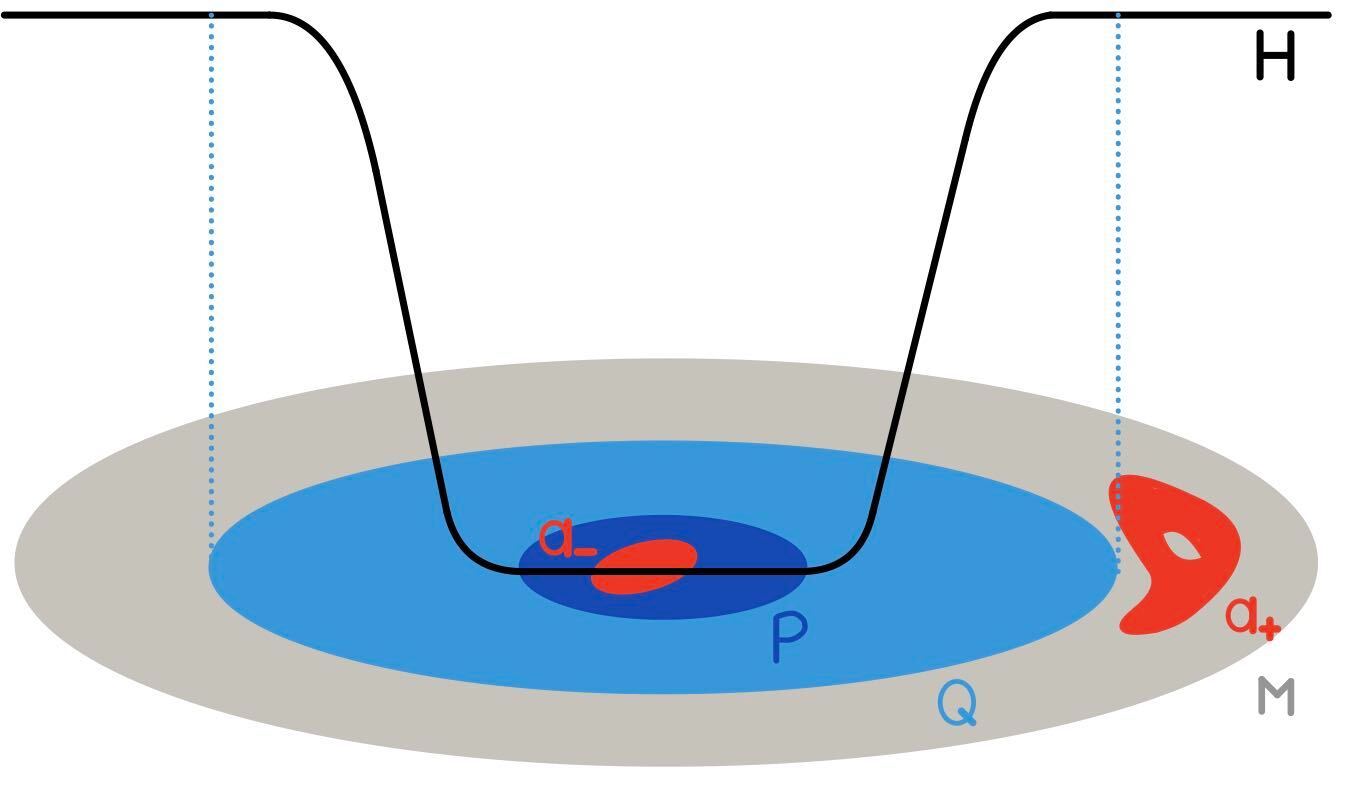}
    \caption{Sketch of a $(a_-,a_+)$-admissible Hamiltonian.}
    \label{figH}
\end{figure}
\end{Definition}
\noindent
Observe that clearly
$$
    C^{(2)}_{\mathrm{HZ}}(M,\omega;a_-,a_+)\leq C^{(2o)}_{\mathrm{HZ}}(M,\omega;a_-,a_+).
$$
We further note this mostly trivial lemma.
\begin{Lemma}\label{lem1}
 If $S$ is some section of $S^2\times \Sigma\to\Sigma$, then the following holds
\begin{itemize}
    \item[(i)] $C^{(2o)}_{\mathrm{HZ}}(S^2\times \Sigma\setminus S,\omega;[pt],[pt])= c_{\mathrm{HZ}}^0(S^2\times \Sigma\setminus S,\omega),$
        \item[(ii)]     $C^{(2o)}_{\mathrm{HZ}}(S^2\times \Sigma\setminus S,\omega;[pt],[pt])\leq C_{\mathrm{HZ}}^{(2o)}(S^2\times \Sigma,\omega,[pt],[S]),$
\end{itemize}
for an arbitrary symplectic form $\omega$.
\end{Lemma}
\begin{proof}
\begin{itemize}
    \item[(i)] follows directly from \cite[Lem. 1.4]{Lu06}.
        \item[(ii)] We need to show that  $\HH_a(S^2\times\Sigma\setminus S; [pt], [pt])\subset \HH_a(S^2\times\Sigma; [pt], [S])$. Take $H\in \HH_a(S^2\times\Sigma\setminus S; [pt], [pt])$ then $H$ satisfies all conditions of $\HH_a(S^2\times\Sigma; [pt], [S])$ trivially, except the fourth. But as $ Q\subset \left( S^2\times \Sigma\right)\setminus S$ by the first condition, it follows that $S\subset \left(S^2\times \Sigma\right)\setminus Q$ as well and thus the fourth condition is satisfied.\qedhere
\end{itemize}
\end{proof}
%\todo[inline]{The other inclusion is at least not as easy I think. The non-trivial step is showing that $H\in \HH_a(S^2\times\Sigma; [pt], [S])$ also satisfies $Q\subset S^2\times\Sigma\setminus S$, but condition four only guaranties that there exists a subset $\Tilde S$ homologous to $S$ such that $Q\subset S^2\times \Sigma\setminus\Tilde S$. }

\subsection{Lu's theorem}\label{sec3.3}
In order to state Lu's theorem we need to collect some information about pseudoholomorphic curves (or in particular spheres) in a four-manifold $M$ (in particular $S^2\times\Sigma$) and Gromov--Witten invariants. We will only focus on the special cases we need for our purpose. The main source for this is \cite{Wdl18} and we will also adopt the notation of C. Wendl, i.e. denote by $\MM_{g,m}([A];J)$ the moduli space of unparametrized $J$-holomorphic curves homologous to $[A]\in H_2(M,\Z)$ of genus $g$ and with $m$ marked points. This space admits an evaluation map
$$
\text{ev}:\MM_{g,m}([A];J)\to M^m
$$
that sends a $J$-holomorphic curve to the image of its marked points and can be used to define the constrained moduli space
$$
\MM_{g,m}([A];J;p_1,\ldots, p_m)=\text{ev}^{-1}(p_1,\ldots, p_m)
$$
for a point $(p_1,\ldots,p_m)\in M^m$. One can show \cite[Thm. 2.12]{Wdl18}) that for a generic $\omega$-tame almost complex structure the subset of somewhere-injective curves
$$
\MM_{g,m}^*([A];J)\subset \MM_{g,m}([A];J)
$$
is a smooth finite dimensional manifold. In general $\MM^*_{g,m}([A];J)$ is non-compact, but for generic $J$ the evaluation map
$$
\text{ev}: \MM^*_{g,m}([A];J)\to M^m
$$
defines a pseudo-cycle \cite[Thm. 7.29]{Wdl18}. The Gromov--Witten invariants are then defined via the homomorphism
$$
\text{GW}_{g,m,[A]}^{(M,\omega)}:H^*(M,\Q)^{\otimes m}\to\Z
$$
that is given by the intersection product of pseudo cycles 
$$
\text{GW}_{g,m,[A]}^{(M,\omega)}(\alpha_1,\ldots,\alpha_m):=[\text{ev}]\cdot[\text{PD}(\alpha_1)\times\ldots\times\text{PD}(\alpha_m)],
$$
where $\mathrm{PD}$ denotes the Poincaré dual. Details and a precise definition can be found in \cite[Def. 7.30]{Wdl18}.
G. Lu found a way to estimate the pseudo symplectic capacities of Hofer--Zehnder type in terms of Gromov--Witten invariants. Even though in G. Lu's paper everything is formulated for arbitrary genus, we will only talk about the genus zero case, as that is the only case we will use.
\begin{Definition}
Let $(M,\omega)$ be a closed symplectic manifold and let $\alpha_-,\alpha_+ \in H^*(M;\Q)$. We define
$$
\mathrm{GW}(M,\omega;\alpha_-,\alpha_+) \in (0,\infty]
$$
as the infimum of the $\omega$-areas $\omega([A])$ of the homology classes $[A] \in H_2(M;\Z)$ for which the Gromov--Witten invariant
$\mathrm{GW}^{(M,\omega)}_{A,0,m+2}(\alpha_-,\alpha_+,\beta_1,\ldots,\beta_m) $ is different from zero
for some cohomology classes $\beta_1,\ldots,\beta_m \in H^*(M;\Q)$ and an integer $m \geq 1$. 
\end{Definition}
\noindent
We have used the convention $\inf\emptyset = \infty$. Compactness of the space of stable $J$-holomorphic maps yields positivity of $\text{GW}(M,\omega;\alpha_-,\alpha_+) $.
To find an upper bound for the Hofer--Zehnder capacity we will use the following theorem by G. Lu  \cite[Thm. 1.10]{Lu06}.
\begin{theorem}\label{thm5}
For any closed symplectic manifold $(M,\omega)$ of dimension at least four and nonzero homology classes $\alpha_-,\alpha_+ \in H^*(M;\Q)$,
$$
C^{(2o)}_{\mathrm{HZ}}(M,\omega;\mathrm{PD}(\alpha_-),\mathrm{PD}(\alpha_+)) \leq \mathrm{GW}(M,\omega;\alpha_-,\alpha_+).
$$
\end{theorem}
\begin{Remark}\label{rem7}In the proof of Theorem \ref{thm1} we will see that $(D_\lambda\Sigma,\omega_s)$ is actually symplectomorphic to the symplectic manifold $(S^2\times\Sigma\setminus S)$ for a suitable symplectic form $\tilde\omega$ and a section $S$.
In that circumstance, we will use the following estimate based on Lemma \ref{lem1} 
$$
c_{\mathrm{HZ}}^0(S^2\times \Sigma\setminus S,\tilde\omega)=C^{(2o)}_{\mathrm{HZ}}(S^2\times \Sigma\setminus S,\tilde\omega;[pt],[pt])
\leq C_{\mathrm{HZ}}^{(2o)}(S^2\times \Sigma,\tilde\omega,[pt],[S]).
$$
\end{Remark}

\subsection{A non-vanishing Gromov--Witten invariant}\label{sec3.1}
For the proof of Theorem \ref{thm1} we will symplectically embed $(D_\lambda \Sigma,\omega_s)$ into $(S^2\times \Sigma, \Tilde{\omega})$, where $\Tilde{\omega}$ is such that the fibers of $S^2\times\Sigma\to\Sigma$ are symplectic. As explained in the last section we can find an upper bound for the Hofer--Zehnder capacity if certain Gromov--Witten invariants of $(S^2\times \Sigma, \Tilde{\omega})$ do not vanish. This is what we will prove now making use of the following proposition that can be found in \cite[Prop. 2.2]{Wdl18}.
\begin{Proposition}
 Suppose $(M,\omega)$ is a symplectic manifold and $A\subset M$ is a smooth 2-dimensional submanifold. Then $A$ is a symplectic submanifold if and only if there exists an $\omega$-tame almost complex structure $J$ preserving $TA$. \qed
\end{Proposition}

\noindent
This implies there exists an almost complex structure $J$ on $T(S^2\times\Sigma)$ that is along a fiber $A=S^2\times\lbrace q\rbrace$ of the form
$$
\begin{pmatrix}
J_0 & 0\\
0 & J_1 
\end{pmatrix}: TS^2\oplus T\Sigma\to TS^2\oplus T\Sigma.
$$
Therefore $J_0=J\vert_{A}$ is a complex structure on $A$ as any almost complex structure on a two-dimensional manifold is integrable. Further any two complex structures on the two-sphere are the same up to biholomorphism, i.e. there exists a biholomorphism 
$$
\varphi: (S^2,\iota)\to (A,J_0).
$$
From there it is easy to see that 
$$
u: S^2\to S^2\times \Sigma;\qquad x\mapsto (\varphi(x),q)
$$
is an embedded $J$-holomorphic sphere. In particular, the self-intersection number of $u$ (denoted by $\delta(u)$) vanishes and an application of the adjunction formula (Thm. 2.51 in \cite{Wdl18}) yields
$$
c_1([u])=c_1([A])=[A]\cdot[A]-2\delta(u)+\chi(u)=2.
$$

%\textcolor{red}{The following lemma from \cite{Wdl18} (lemma 1.26) gives some more information. For this we need to restrict to a closed symplectic four manifold $(M,\omega)$ and $A$ a symplectic sphere with $[A]^2=0$.
%\begin{Lemma}
% After a generic perturbation of $J$, the component $\MM_{[A]}(J)$ is a nonempty, smooth, compact, oriented 2-dimensional manifold whose elements are each embedded $J$-holomorphic spheres with pairwise disjoint images, foliating an open subset of $M$. 
%\end{Lemma}
%\noindent
%As $\MM_{[A]}(J)$ is a compact manifold the Gromov--Witten invariant reduces to the intersection product
%$$
%\text{GW}_{[A],0,1}(\text{PD}(q,p))=\text{ev}_*[\MM_{[A]}(J)]\cdot[\text{PD}(q,p)]
%$$}
\noindent
If we now take any somewhere injective $v\in \MM_{0,1}([A];J)$  the adjunction formula tells us that 
$$
\delta(v)=\frac{1}{2}([A]\cdot[A]+\chi(S^2)-c_1([A])=0,
$$
thus all curves in $\MM^*_{0,1}([A];J)$ are embedded. We further observe that the index of any $v\in\MM_{0,1}([A];J)$ is given by
$$
\text{ind}(v)=2 c_1([v])-\chi(v)=2
$$
and we can apply automatic transversality (\cite[Thm. 2.46]{Wdl18}) to find that $\MM^*_{0,1}([A];J)$ is a manifold.\\
Next, we fix $p\in S^2$. Then the embedded $J$-holomorphic sphere $u$ we constructed earlier lies in $\MM^*_{0,1}([A];J; (p,q))$, but on the other hand by positivity of intersections (\cite[Thm. 2.49]{Wdl18}) there can not be a different somewhere-injective curve $v\in\MM_{0,1}([A];J;(p,q))$. 
Therefore no other somewhere-injective $J$-holomorphic sphere of class $[A]$ intersects $u$ and in particular we have found the following corollary.
\begin{Corollary}\label{GW} Let $(p,q)\in S^2\times \Sigma$ and $A=S^2\times\lbrace q\rbrace$ then
$$
  \mathrm{GW}_{[A],0,1}((p,q))=\pm 1\neq 0 
 $$ 
depending on orientations. 
\end{Corollary}
\noindent
Let us mention a consequence of Corollary \ref{GW} that will be used later. If $[S]$ is the homology class of some section $S$ of the trivial bundle $M=S^2\times\Sigma\to\Sigma$ we find with the help of \cite[Exercise 7.17]{Wdl18} 
\begin{equation}\label{eq5}
\text{GW}_{[A],0,3}((p,q),\text{PD}[S],\text{PD}[S])=\left([A]\cdot [S]\right)^2\text{GW}_{[A],0,1}(\text{PD}(p,q))\neq 0.
\end{equation}

\begin{Remark}\label{rem1}
By Theorem \ref{thm5} we have for any fiber $A$ of $S^2\times\Sigma\to\Sigma$,
$$
C_{\mathrm{HZ}}^{(2o)}(S^2\times \Sigma,\tilde\omega;[pt],[S])\leq \mathrm{GW}(S^2\times \Sigma,\tilde\omega;\mathrm{PD}(pt),\mathrm{PD}(S))
\leq \tilde\omega([A])
$$
as the Gromov--Witten invariant \eqref{eq5} does not vanish and the fibers are symplectic. Taking Remark \ref{rem7} into account, we conclude that if $\tilde\omega$ is a symplectic form on $S^2\times \Sigma$ such that the fibers are symplectic, we find an upper bound of the Hofer--Zehnder capacity of $(S^2\times\Sigma\setminus S,\tilde \omega)$ by the $\tilde\omega$-area of the fibers. 
\end{Remark}
\section{Proof of Theorem 1}\label{sec4}
We find a lower bound for the Hofer--Zehnder capacity by constructing an explicit admissible Hamiltonian and show that the lower bound is also an upper bound by using Lu's theorem.

\subsection{Lower bounds}
We will use the kinetic Hamiltonian $E(x,v)=\frac{1}{2}\vert v\vert^2 $ to find a lower bound of the Hofer--Zehnder capacity of $(D_\lambda\Sigma,\omega_s)$. We need to look for periodic solutions $\gamma(t)=(x(t),v(t))$ to 
\begin{equation}\label{e3}
    \dot\gamma=X_E=X+sV.
\end{equation}
Applying $\dd\pi$ to this equation yields
$$
\dot x=\dd\pi\dot\gamma\overset{\eqref{e3}}{=}\dd\pi (X+sV)=\dd\pi\LL^H(v)=v.
$$
Thus our solutions must be of the form $\gamma(t)=(x(t),\dot x(t))$.\\
On the other hand applying the projection $\mathcal{P}: TT\Sigma\to\mathcal{V}$ on the vertical bundle yields 
$$
\nabla_{\dot x}\dot x=\mathcal{P}(\dot\gamma)\overset{\eqref{e3}}{=}\mathcal{P}(X+sV)=s\iota v.
$$
This means the projection to $M$ of solutions are curves of geodesic curvature $\kappa_g=\frac{s}{\vert v\vert}$. If $R$ denotes the radius (with respect to the Riemannian metric $g$) of a geodesic circle we know using normal polar coordinates that its circumference $C$ and the geodesic curvature $\kappa_g$ are
\begin{align*}
   C&=\frac{2\pi}{\sqrt{\kappa}}\sin(\sqrt{\kappa}R)=\frac{2\pi\sqrt{\kappa}^{-1}\tan(\sqrt{\kappa}R)}{\sqrt{1+(\tan(\sqrt{\kappa}R))^2}},\\ 
   \kappa_g&=\frac{\sqrt{\kappa}}{\tan(\sqrt{\kappa}R)}.
\end{align*}
Inserting $\kappa_g$ into $C$ yields
$$
C=\frac{2\pi}{\kappa_g\sqrt{1+\kappa/\kappa_g^2}}=\frac{2\pi\vert v\vert}{\sqrt{s^2+\kappa\vert v\vert^2}},
$$
where in the last step we inserted $\kappa_g=s/\vert v\vert$. Now, we conclude that the period is given by
$$
T=\frac{C}{\vert v\vert}=\frac{2\pi}{\sqrt{s^2+\kappa\vert v\vert^2}}.
$$
We can find a function $h:\R\to\R$ such that all solutions of the Hamiltonian system belonging to $H=h\circ E$ have period one. The periods $T_E, T_H$ belonging to the Hamiltonians $E,H$ are related via
$$
 T_H=\frac{T_E}{h'(E)}
$$
and therefore if we set
$$
H(E)=\left\{\begin{array}{ll} \frac{2\pi}{\kappa}\sqrt{s^2+2\kappa E}, & \text{for}\ \kappa\neq 0, \\
         \frac{2\pi E}{ s}, & \text{for}\ \kappa=0.\end{array}\right.
$$
all solutions for $X_H$ have period one. We have now found a nice Hamiltonian but to find a lower bound of Hofer--Zehnder capacity we need to modify $H$ such that $H$ becomes admissible. This can be done with the help of a function
$f:[a,b]\to [0,\infty)$ satisfying
$$
\begin{aligned}
&0\leq f'(x)< 1, \\
&f(x)=0\ \  \text{near}\ \ a,\\
&f(x)=b-a-\varepsilon\ \  \text{near}\ \ b
\end{aligned}
$$
with $a=\min H$ and $b=\max H$. Then all solutions to the Hamiltonian system with Hamiltonian $\tilde H=f\circ g\circ E$ have period 
$$
T=\frac{1}{f'(g(E))}> 1.
$$
Thus $\tilde H$ is admissible and we find the estimate
$$
c_{\mathrm{HZ}}(D_\lambda\Sigma,\omega_s)\geq b-a=\left\{\begin{array}{ll} \frac{2\pi}{\kappa}\left( \sqrt{s^2+\kappa\lambda^2}-s\right) & \text{for}\ \kappa\neq 0, \\
         \frac{\pi\lambda^2}{ s} & \text{for}\ \kappa=0.\end{array}\right.
$$

\subsection{Upper bound}
The idea is to use Lu's theorem. We first find a symplectomorphism
$$
F:(D_0 ^\lambda\Sigma,\omega_s)\to (D_a^b\Sigma, \dd((E+s/\kappa)\tau))
$$
for some $a,b>0$ depending on $s,\lambda,\kappa$, where $D_{a}^b\Sigma=\lbrace (x,v)\in T\Sigma\vert\ a<\vert v\vert<b\rbrace $. Observe that 
$$
\dd ((E+s/\kappa)\tau)=\dd E\wedge \tau-(E\kappa+s)\pi^*\sigma
$$
makes the fibers symplectic due to the first summand. The next step is to fiberwise compactify $(D_a^b\Sigma, \dd((E+s/\kappa)\tau))$. We will find that this yields the trivial sphere bundle $S^2\times \Sigma$ with a symplectic structure $\Tilde{\omega}$ that is induced by $\dd((E+s/\kappa)\tau)$ and therefore makes the fibers symplectic. As discussed in Section 3 this means that the symplectic area of these fibers yields an upper bound to the Hofer--Zehnder capacity.

\subsubsection{Symplectization of the fibers}
Finding $F$ works as follows. Recall that for $\kappa\neq 0$ 
$$
\omega_s=\dd\alpha-s\pi^*\sigma=\dd\left(\alpha+\frac{s}{\kappa}\tau\right)
$$
on $D_0^\lambda\Sigma$ and that the $\alpha$ term makes the fibers Lagrangian instead of symplectic. Further, the trajectories of $E$ are geodesic circles. By mapping a point $\gamma(t)$ on a trajectory to the circle center, we eliminate the horizontal movement of the trajectories and therefore the $\dd\alpha$ part of $\omega_s$. More precisely this is done by taking the flow $\phi_b$ of $H$, which as shown in Lemma \ref{lem4} couples $\alpha$ and $\tau$. 
We make the ansatz $F_s=m_{a_s(r)}\circ \phi_{b_s(r)}$, where $r=r(x,v)=\vert v\vert$ and $m_a$ scales the fibers as defined in Definition \ref{def}. The following calculation in the case of $\kappa=1$ was done in \cite[Thm. A1]{BR19} and we only adapt it to arbitrary $\kappa$. Observe that 
\begin{equation}\label{eq6}
   \dd F_s=\frac{1}{b_s(r)}Y\otimes\dd(a_s(r))-a_s(r)\dd m_{a_s(r)}\cdot H\otimes \dd (b_s(r))+\dd m_{a_s(r)}\dd\phi_{b_s(r)} 
\end{equation}
and $\dd m_a\cdot H= a\cdot H$, thus $\tau$ vanishes on the first two terms. By Lemma \ref{lem4} and \ref{lem3}
\begin{align*}
    F_s^*\left (\frac{r^2}{2}+\frac{s}{\kappa}\right )\tau&=\left ( \frac{a_s(r)^2r^2}{2}+\frac{s}{\kappa}\right )\left ( -\frac{\sqrt{\kappa}}{r}\sin(r\sqrt{\kappa}b_s(r))\alpha+\cos(r\sqrt{\kappa}b_s(r))\tau\right )\\
    &\stackrel{!}{=}\alpha+\frac{s}{\kappa}\tau.
\end{align*}
We read off
\begin{align*}\label{e1}
    \left ( \frac{a_s(r)^2r^2}{2}+\frac{s}{\kappa}\right )\sin(r\sqrt{\kappa}b_s(r))&=\frac{-r}{\sqrt{\kappa}},\\
    \left ( \frac{a_s(r)^2r^2}{2}+\frac{s}{\kappa}\right )\cos(r\sqrt{\kappa}b_s(r))&=\frac{s}{\kappa},
\end{align*}\label{eq1}
which is solved by
\begin{align}
    b_s(r)=\frac{1}{r\sqrt{\kappa}}\textrm{arctan}\left(-\frac{r\sqrt{\kappa}}{s}\right),\qquad
    a_s(r)=\frac{1}{r}\sqrt{2\left (\sqrt{\frac{r^2}{\kappa}+\frac{s^2}{\kappa^2}}-\frac{s}{\kappa}\right )}.
\end{align}
Observe that $F_s$ is a symplectomorphism onto its image as $(ra_s(r))'\neq 0$ for all $0<r<\lambda$. The image of $F_s$ is $D_a^b\Sigma$, where for $\kappa>0$ we have
\begin{equation}\label{eq10}
    a=\lim_{r\to 0}(ra_s(r))=0,\qquad b=\lim_{r\to \lambda}(ra_s(r))=\sqrt{2\left(\sqrt{\frac{\lambda^2}{\kappa}+\frac{s^2}{\kappa^2}}-\frac{s}{\kappa}\right)}
\end{equation}
whereas for $\kappa<0$ we have
\begin{equation}\label{eq11}
    a=\lim_{r\to \lambda}(ra_s(r))=\sqrt{2\left(\sqrt{\frac{\lambda^2}{\kappa}+\frac{s^2}{\kappa^2}}-\frac{s}{\kappa}\right)},\qquad b=\lim_{r\to 0}(ra_s(r))=2\sqrt{-\frac{s}{\kappa}}.
\end{equation}

%\paragraph{Sphere}\ \\
%The case of the sphere i.e. $\kappa>0$ was discussed by Benedetti and Ritter in \cite[thm. A1]{BR19}, and they found that
%$$
%F_s: (D_0^ \lambda S^2,\omega_s)\to (D_0^{\sqrt{2(\sqrt{\lambda^2/\kappa+s^2/\kappa}-s/\kappa)}}S^2, \dd((E+s)\tau))
%$$
%is a symplectomorphism. 
%\paragraph{Hyperbolic surfaces}\ \\
%The case $\kappa<0$ and $\sqrt{-\kappa}\lambda < s$ works very similar. 
%\begin{Proposition}
%If $\kappa<0$ and $\sqrt{-\kappa}\lambda < s$, there is a symplectomorphism 
%$$
%F: (D^\lambda_0\Sigma,\omega_s)\mapsto(D_{\sqrt{2(\sqrt{s^2/\kappa^2+\lambda^2/\kappa}-s/\kappa)}}^{2\sqrt{-s/\kappa}}\Sigma, \dd (E+s/\kappa)\tau),
%$$
%where $D_{a}^b\Sigma=\lbrace (x,v)\in T\Sigma\vert\ a<\vert v\vert<b\rbrace $.
%\end{Proposition}

%\begin{proof}
%The proof is completely analogous to the proof of theorem A.1. in \cite{BR19}. To simplify the calculations we assume $\kappa=-1$, the general case works analogous. The functions $a_s, b_s$ of $r=\vert v\vert$ are then given by
%\begin{align*}
%    a_s(r)&=\frac{1}{r}\sqrt{2(\sqrt{s^2-r^2}+s)}=\sqrt{\frac{2}{s-\sqrt{s^2-r^2}}}\\
%    b_s(r)&=\frac{1}{r}\tanh^{-1}\left (\frac{r}{s}\right )=\frac{1}{2r}\ln\left( \frac{s+r}{s-r}\right ).
%\end{align*}
%We observe that 
%$$
 %   a_s: (0,s)\to (\sqrt{2/s},\infty )
%$$
%is strictly monotone decreasing and
%$$
%    b_s:[0,s)\to [1/s,\infty) 
%$$
%is strictly monotone increasing. Therefore $F$ is well defined on $D_\lambda^0\Sigma$ for all $\lambda\leq s$ and a diffeomorphism.
%\end{proof}

%\paragraph{Torus}\ \\
\noindent
For $\kappa= 0$ the calculations in the beginning of this section do not hold, but the geometric idea still works as long as $s\neq 0$. Approximation of the equations \eqref{e1} for small $\kappa$ yields
\begin{align*}
    &\left(\frac{a_s(r)^2r^2}{2}+\frac{s}{\kappa}\right)\left(r\sqrt{\kappa} b_s(r)+\mathcal{O}(\sqrt{\kappa}^3) \right)=\frac{-r}{\sqrt{\kappa}},\\
    &\left(\frac{a_s(r)^2r^2}{2}+\frac{s}{\kappa}\right)\left( 1-\frac{r^2\kappa b_s(r)^2}{2}+\mathcal{O}(\sqrt{\kappa}^3)\right)=\frac{s}{\kappa}.
\end{align*}
Taking the limit $\kappa\to 0$ in the second equation now yields
\begin{align*}
    sb_s(r)&=-1,\\
    a_s(r)^2&=sb_s(r)^2.
\end{align*}
thus $a_s(r)=\frac{1}{\sqrt{s}}$ and $b_s(r)=-\frac{1}{s}$. On the torus the coordinate description of  $F_s=m_{1/\sqrt{s}}\circ\phi_{-1/s}$ is simply given by
$$
F_s:D_\lambda \T^2\to D_{\lambda/\sqrt{s}}\T^2;\ \left( x,y,\frac{v_x}{s},\frac{v_y}{s}\right)\mapsto\left(x-\frac{v_y}{s},y+\frac{v_x}{s},\frac{v_x}{\sqrt{s}},\frac{v_y}{\sqrt{s}}\right)
$$
and 
\begin{align*}
    F_s^*(dE\wedge\tau-s\sigma)&=F_s^*\lbrack \frac{1}{2E} (v_x\dd v_x+ v_y\dd v_y)\wedge (-v_y\dd v_x+v_x\dd v_y)-s\dd x\wedge\dd y\rbrack\\
    &=F_s^*(\dd v_x\wedge\dd v_y-s\dd x\wedge\dd y)\\
    &=\dd v_x\wedge\dd x+ \dd v_y\wedge \dd y-s\dd x\wedge\dd y\\
    &=\dd\alpha-s\pi^*\sigma.
\end{align*}
Thus, $F_s$ is also a symplectomorphism in the case of $\kappa=0$. This finishes the proof of the following lemma by Ginzburg \cite[Lem. 5.3]{Gb96}.
\begin{Lemma}
We can identify
$$
(D_\lambda\T^2,\omega_s)\cong(D_{\lambda/\sqrt{s}}\times\T^2,\dd E\wedge \tau-s\pi^*\sigma).
$$
whenever the magnetic field does not vanish, i.e. $s\neq 0$.\qed
\end{Lemma}
\subsubsection{Compactification}
We want to compactify $(D_a^b\Sigma,\tilde\omega:= \dd E\wedge\tau-(\kappa E+s)\pi^*\sigma)$ fiberwise. 
All closed surfaces $\Sigma$ can be embedded into $\R^3$. Thus adding the normal bundle to the tangent bundle yields the trivial bundle $\R^3\times\Sigma$. We conclude that the 2-point compactification, which can be seen as the 2-sphere subbundle of $\R^3\times \Sigma$ is trivial as well.  

\noindent
We will now make the compactification of $D_a^b\Sigma$ more explicit to show that (i) the symplectic form $\tilde\omega$ extends to a symplectic form on $S^2\times\Sigma$ and $(ii)$ the map $F_s$ extends symplecticly to the zero-section. Consider first the map 
$$
\Upsilon: D^b_a\Sigma\to D_0^{\sqrt{b^2-a^2}}\Sigma;\qquad (x,v)\mapsto \left(x,\frac{\sqrt{r^2-a^2}}{r}v\right).
$$
It is enough to compactify $D^\rho_0\Sigma$, where $\rho=\sqrt{b^2-a^2}$. The lower boundary can be compactified by the inclusion $i:D_0^\rho\Sigma\hookrightarrow D_\rho\Sigma$, the upper boundary using the composition $j\circ i$, where 
$$
j: D^\rho_0\Sigma\to D^\rho_0\Sigma;\qquad (x,v)\mapsto \left( x,\frac{\sqrt{\rho^2-r^2}}{r}v\right)
$$
flips the boundaries. All in all the compactification is given by the maps 
$$
i\circ\Upsilon: D_a^b\Sigma\to D_{\sqrt{b^2-a^2}}\Sigma \qquad \text{and}\qquad i\circ j\circ\Upsilon: D_a^b\Sigma\to D_{\sqrt{b^2-a^2}}\Sigma. 
$$
The pushforward of $\tilde\omega=\dd E\wedge\tau-(\kappa E+s)\pi^*\sigma$ is respectively given by
$$
r\dd r\wedge \tau -\left(\frac{\kappa}{2}(a^2+r^2)+s\right)\pi^*\sigma\qquad\text{and}\qquad -r\dd r\wedge \tau -\left(\frac{\kappa}{2}(b^2-r^2)+s\right)\pi^*\sigma.
$$
If these forms extend symplecticly to $r=0$, the symplectic form $\tilde\omega$ extends to a symplectic form on $S^2\times\Sigma$. This is the case if and only if $\frac{\kappa}{2}a^2+s\neq 0$ and $\frac{\kappa}{2}b^2+s\neq 0$ and indeed plugging in equations \eqref{eq10} and \eqref{eq11} yields
$$
\frac{\kappa}{2}a^2+s=\left\{\begin{array}{ll} s & \text{for}\ \kappa\geq 0 \\
         \sqrt{\kappa\lambda^2+s^2} & \text{for}\ \kappa<0\end{array}\right. ,\qquad
\frac{\kappa}{2}b^2+s=\left\{\begin{array}{ll} \sqrt{\kappa\lambda^2+s^2} & \text{for}\ \kappa\geq 0 \\
         -s & \text{for}\ \kappa<0.\end{array}\right.
$$
We denote the extension of the symplectic form to $S^2\times\Sigma$ also by $\tilde\omega$.
\begin{Lemma}
The symplectic embedding 
$$ \ (D_0^\lambda\Sigma,\omega_s)\overset{F_s}{\longrightarrow} (D_a^b\Sigma,\dd E\wedge\tau-(E\kappa+s)\pi^*\sigma)\hookrightarrow(S^2\times\Sigma,\tilde\omega)$$
also denoted by $F_s$ extends smoothly to the zero-section.
\end{Lemma}
\begin{proof}
If $\kappa\geq 0$, then $a=0$, $\Upsilon=\text{id}$ and we need to extend $i\circ F_s: D_0^\lambda\Sigma\to D_b\Sigma$ to the zero section. This is possible as $a_s$ and $b_s$ extend smoothly to the zero-section.\\
If $\kappa<0$, then we need to extend
$$
\tilde F=i\circ j\circ \Upsilon\circ F_s: D_0^\lambda\Sigma\to D^{\sqrt{b^2-a^2}}\Sigma
$$
to the zero section. We have $\tilde F= m_{\tilde a_s}\circ \phi_{b_s}$, where
$$
\tilde a_s(r)=\frac{1}{r}\sqrt{b^2-a_s(r)^2r^2}=\frac{1}{r}\sqrt{-2\left(\frac{s}{\kappa}+\sqrt{\frac{r^2}{\kappa}+\frac{s^2}{\kappa^2}}\right)}.
$$
We see that as in the case $\kappa\geq 0$ the function $\tilde a_s$ extends smoothly to the zero-section.
\end{proof}
\subsubsection{Application of Lu's theorem}

We find ourself in the setup described in Section \ref{sec:3}, having a symplectic embedding
$$
F_s: (D_\lambda\Sigma,\omega_s)\hookrightarrow (S^2\times\Sigma,\tilde\omega).
$$
The map $F_s$ hits everything but a section $\mathcal{S}:\Sigma\to S^2\times\Sigma$ which can be identified with the section antipodal to the image of the zero-section under $F_s$ and we conclude
$$
c_{\mathrm{HZ}}^0(D_\lambda \Sigma,\omega_s)= c_{\mathrm{HZ}}^0(S^2\times\Sigma\setminus \mathcal{S},\tilde\omega).
$$
As derived in Remark \ref{rem7} and Remark \ref{rem1} an upper bound to the Hofer--Zehnder capacity is given by the symplectic area of a fiber, i.e. let $A\subset S^2\times\Sigma$ be a fiber, then 
$$
c_{\mathrm{HZ}}^0(D_\lambda \Sigma,\omega_s)\leq\tilde\omega(A).
$$
Thus all that is left to do now is to compute this symplectic area
$$
\omega(A)=\int_0^{2\pi}\dd\varphi\int_a^b r\dd r=\left\{\begin{array}{ll} \frac{2\pi}{\kappa}\left( \sqrt{s^2+\kappa\lambda^2}-s\right), & \text{for}\ \kappa\neq 0, \\
         \frac{\pi\lambda^2}{s}, & \text{for}\ \kappa=0.\end{array}\right.
$$
which is the same as the lower bound we computed in Section 4.1 and therefore we conclude
$$
c_{\mathrm{HZ}}(D_\lambda \Sigma,\omega_s)=c_{\mathrm{HZ}}^0(D_\lambda\Sigma,\omega_s)=\left\{\begin{array}{ll} \frac{2\pi}{\kappa}\left( \sqrt{s^2+\kappa\lambda^2}-s\right), & \text{for}\ \kappa\neq 0, \\
         \frac{\pi\lambda^2}{s}, & \text{for}\ \kappa=0.\end{array}\right.
$$

\section{Proof of Theorem 3}\label{sec5}
To ease the notation we only consider the case of a hyperbolic surface of constant curvature $\kappa=-1$. For weak magnetic fields, i.e. $s<\lambda$, we can calculate the Hofer--Zehnder capacity relative to $D_s\Sigma$. The value of the relative capacity was given in Ginzburg \cite[Ex. 4.7]{Gb04} together with an idea of the proof. Applying the ideas of the present paper we can fill in the details. 
\begin{theorem}
If $s<\lambda$, there is a symplectomorphism
$$
F: (D_s^\lambda\Sigma,\omega_s)\to \left(D_0^{\sqrt{\lambda^2-s^2}}\Sigma,\omega_0\right).
$$
\end{theorem}
\begin{proof}
The argument works analogously to the proof of Proposition 1.5 in \cite{Bd16}. We make the ansatz $F=m_{a_s(r)}\circ \phi_{b_s(r)}$ for some real valued functions $a_s, b_s$ of $r$. The first two terms in equation \eqref{eq6} also annihilate $\alpha$ and therefore we obtain by Lemma \ref{lem4} and \ref{lem3} 
\begin{align*}
    F^*\alpha&=\phi^*_{b_s(r)}m^*_{a_s(r)}\alpha=a_s(r)(\cosh(rb_s(r))\alpha+r\sinh(rb_s(r))\tau)\\
    &\stackrel{!}{=}\alpha-s\tau.
\end{align*}
We read off
\begin{align*}
    a_s(r)\cosh(rb_s(r))&=1,\\
    a_s(r)\sinh(rb_s(r))&=-\frac{s}{r}.
\end{align*}
This is for $r>s$ solved by the smooth functions
\begin{align*}
    a_s(r)&=\frac{\sqrt{r^2-s^2}}{r},\\
    b_s(r)&=\frac{\tanh^{-1}(-\frac{s}{r})}{r}=\frac{1}{2r}\ln\left(\frac{r-s}{r+s}\right)
\end{align*}
We observe that $(ra_s(r))'\neq 0$ for all $r>s$ and therefore $F$ is a diffeomorphism.
%\begin{align*}
%   & a_s: (s,\infty)\to (0, 1)\\
%   & b_s: (s,\infty)\to (-\infty,0) 
%\end{align*}
%are strictly monotone increasing and therefore $F$ is a diffeomorphism.
\end{proof}
\noindent
It follows that for a fixed free homotopy class of loops $\nu$ we have
\begin{align*}
    c^\nu_{\mathrm{HZ}}(D_\lambda\Sigma,D_s\Sigma,\omega_s)&=c^\nu_{\mathrm{HZ}}\left(D_{\sqrt{\lambda^2-s^2}}\Sigma,\Sigma,\omega_0\right)\\
    &=\sqrt{\lambda^2-s^2}\ c_{\mathrm{HZ}}^\nu(D_1\Sigma,\Sigma,\omega_0)=\sqrt{\lambda^2-s^2}\ l_\nu
\end{align*}
by (\cite[Thm. 4.3]{Wbr06}) as already mentioned in the introduction (see (\ref{rell})).
\section{Symmetries of the magnetic sphere}\label{sec6}
In this section we explicitly build the symplectomorphism in Theorem \ref{thm4}.  It is constructed using the $\text{SO}(3)$ symmetry of $D_\lambda S^2$. The group $\text{SO}(3)$ acts on $S^2$ by isometries and thus induces a symplectic $\text{SO}(3)$ action on the disc bundle $(D_\lambda S^2,\omega_s)$. This action admits a moment map:
$$
\mu_{D_\lambda S^2}(x,v)= x\times v -sx,
$$
where we identified the Lie algebra of $\text{SO(3)}$ with $\R^3$ and $\times$ denotes the cross product.\\
The target space of our embedding will be $S^2\times S^2$ with the split symplectic form $R_1\sigma\oplus R_2\sigma$ for suitable constants $R_1,R_2$. The diagonal action of $\text{SO}(3)$ on this manifold is also symplectic and admits a moment map:
$$
\mu_{S^2\times S^2}(N,S)= R_1 N+R_2S.
$$
This action restricts to an action on $S^2\times S^2\setminus \Delta$ as it leaves the diagonal $\Delta\subset S^2\times S^2$ invariant.
\begin{Theorem}
There is a symplectomorphism, which is equivariant with respect to the Hamiltonian $\mathrm{SO}(3)$ actions, 
$$
F: (D_\lambda S^2,\omega_s)\to (S^2\times S^2\setminus \Delta, R_1\sigma\oplus R_2\sigma),
$$
where $\Delta\subset S^2\times S^2$ denotes the diagonal and $R_1,R_2$ are determined by 
$$
s=R_2-R_1,\ \ \lambda=2\sqrt{R_1R_2}.
$$
More precisely the moment maps are related via
\begin{equation}
    \label{eq:1}
\mu_{D_\lambda S^2}=\mu_{S^2\times S^2}\circ F.
\end{equation}
\end{Theorem}
\noindent
Observe that 
$$
m_{R_1^{-1}}: \  \left(D_{2\sqrt{R_1R_2}}S^2,\dd\alpha-(R_2-R_1)\pi^*\sigma\right)\to\left(D_{2\sqrt{R_2/R_1}}S^2,R_1\left(\dd\alpha-\left(\frac{R_2}{R_1}-1\right)\pi^*\sigma\right)\right)
$$
is an equivariant symplectomorphism. Therefore, we can restrict to the case $R_1=1$ and $R_2=R$.\\
For the construction of $F$ we will need the antipodal map
$$ I:\ S^2\to S^2;\qquad x\mapsto-x$$
and the rotation in the fibers by $90^\circ$ (i.e. complex structure) we defined in Section \ref{sec2}:
\begin{equation}\label{def:2}
\iota: TS^2\to TS^2;\qquad (x,v)\mapsto (x,\iota_xv).
\end{equation}
Let us now denote by $\Psi_c:\ D_\lambda S^2\to D_\lambda S^2$ the geodesic flow for time $c$, i.e. 
$$
\Psi_c(x,v)=(\gamma_{(x,v)}(c),\dot{\gamma}_{(x,v)}(c))
$$
with $\gamma_{(x,v)}(c):=\exp_x(c\vert v\vert)$. If $c:(-\lambda,\lambda)\to(0,\infty)$ is a smooth even function, we can define the smooth map
$$
\varphi_c:D_\lambda S^2\to S^2;\ (x,v)\mapsto\gamma_{(x,v)}(c(\vert v\vert)).
$$
wich is the geodesic flow for time $c(\vert v\vert)$ projected to $\Sigma$. 
Our claim is now that (for suitable smooth even functions $c_1, c_2$)
$$ F: (D_{2\sqrt{R}} S^2,\omega_{R-1})\to (S^2\times S^2\setminus\Delta,\sigma\oplus R\sigma);\ \ \ (x,v)\mapsto (\varphi_{c_1}(x,-\iota_xv),\varphi_{c_2}(I(x),-\iota_{I(x)}\dd I_x v))
$$
is a symplectomorphism. We will spend the rest of this section proving this claim. 
From the definition of $F$ we see that
$$
F(x,v)=(\gamma_{(x,-\iota_x v)}(c_1),\gamma_{(I(x),-\iota_{I(x)}\dd I_x  v)}(c_2)),
$$
\noindent
and we will determine $c_1,\ c_2$ imposing the relation of the moment maps given by equation \eqref{eq:1}. Choosing the coordinate system of $\R^3$ such that the first axis is parallel to $x\times v$ and the second axis is parallel to $x$ and setting $r=\vert v\vert$, we can rewrite the relation \eqref{eq:1} as follows

   \begin{equation}\label{eq:22}
   \begin{aligned}
   &N_1+RS_1=r  &\Leftrightarrow &\ \ \ \sin(c_1r)+R\sin(c_2r)=r, \\
   &N_2+RS_2=1-R &\Leftrightarrow &\ \ \ \cos(c_1r)-R\cos(c_2r)=1-R,
   \end{aligned}
   \end{equation}
\noindent
where $F(x,v)=(N,S)$ and $N_i, S_i$ denote the $i$th component of $N,S$. One can geometrically determine the functions $c_1, c_2$. This is shown in Figure \ref{fig:8} at the end of the paper.
\begin{Lemma}
The functions 
$$
c_1, c_2:\ (-\lambda,\lambda)\to (0,\infty)
$$
implicitly defined via the equation \eqref{eq:22} are smooth and even.
\end{Lemma}
\begin{proof}
Applying the implicit function theorem directly to the equations \eqref{eq:22} yields smoothness of $c_1, c_2$ whenever $r\neq 0$. Further, they are even, as the defining equations \eqref{eq:22} are invariant under $r\to -r$. For the case $r\to 0$ we rewrite \eqref{eq:22} in terms of the function
$$
G: (-\lambda,\lambda)\times(0,\infty)^2\to \R^2;\qquad (r, c_1,c_2)\to \begin{pmatrix}
\tau(c_1r)c_1+R\tau(c_2r)c_2-1\\
\sigma(c_1r)c_1^2-R\sigma(c_2r)c_2^2
\end{pmatrix},
$$
where $\tau,\sigma: \R\to\R$ are the smooth functions given by
$$
\tau(x)=\frac{\sin(x)}{x},\qquad\sigma(x)=\frac{1-\cos(x)}{x^2}.
$$
The equations \eqref{eq:22} are now equivalent to $G(r,c_1,c_2)=0$, taking the derivative in $c_1, c_2$ yields
$$
\dd_{c_1,c_2}G=\begin{pmatrix}
r\tau'(c_1r)c_1+\tau(c_1r) & Rr\tau'(c_2r)c_2+R\tau(c_2r)\\
r\sigma'(c_1r)c_1^2+2\sigma(c_1r)c_1 & -Rr\sigma'(c_2r)c_2^2-2R\sigma(c_2r)c_2.
\end{pmatrix}
$$
Observe that taking the limit $r\to 0$ in \eqref{eq:22} yields
$$
    c_1(0)+Rc_2(0)=1,\qquad
    c_1(0)^2-Rc_2(0)^2=0
$$
thus
$$
c_1(0)=(1+\sqrt R)^{-1},\qquad
c_2(0)=(\sqrt R+ R)^{-1}.
$$
Further, 
$\tau(0)=1=2\sigma(0),\ \tau'(0)=0=\sigma'(0)$ and therefore
$$
\dd_{c_1,c_2}G\vert_{r=0}=\begin{pmatrix}
1 & R\\
c_1(0) & -Rc_2(0)
\end{pmatrix}.
$$
We see that 
$$
\det\left(\dd G_{(c_1,c_2)}\vert_{r=0}\right)=-R(c_1(0)+c_2(0))\neq 0
$$
and it follows by the implicit function theorem that $c_1, c_2$ are smooth at $r=0$.
\end{proof}
\noindent
That $F$ is indeed bijective is shown in figure \ref{fig:10} where an inverse to $F$ is constructed geometrically. We have shown that $F$ is smooth, bijective and $\dd F_q$ is invertible for all $q\in D_\lambda S^2$. Thus, $F$ is a diffeomorphism. It remains to be shown that $F$ is symplectic.
We will use the group action to show that $F$ is a symplectomorphism.
\begin{Lemma}\label{lem5}
For all $(x,v)\in D_\lambda S^2$ and all $g\in\mathrm{SO}(3)$ we have
$$
g\cdot F(x,v)=F(g\cdot (x,v)),
$$
where $g\cdot$ denotes the group action on $D_\lambda S^2$ or $S^2\times S^2\setminus\Delta$.
\end{Lemma}
\begin{proof}
For any group element $g\in \text{SO}(3)$ write $\psi_g:S^2\to S^2$ for the group action. We observe that by linearity
$$
I\circ\psi_g=\psi_g\circ I.
$$
As $\text{SO}(3)$ acts by orientation preserving isometries 
$$
 \iota\circ\dd\psi_g=\dd\psi_g\circ\iota
$$
and further if $\gamma_{(x,v)}$ is the geodesic starting at $x\in S^2$ in direction $v\in T_x S^2$ then
$$
\psi_g(\gamma_{(x,v)}(t))=\gamma_{(\psi_g(x),\dd(\psi_g)_xv}(t).
$$
Taking all three relations together the claim follows.
\end{proof}
\noindent
Given any generator $\xi\in\mathfrak{so}(3)$ we denote by
\begin{align*}
  \rho(\xi)^{D_\lambda S^2}_{(x,v)}&=\diff \exp(t\xi)\cdot (x,v), \\
  \rho(\xi)^{S^2\times S^2}_{(N,S)}&=\diff \exp(t\xi)\cdot (N,S)
\end{align*}
the vector field on $D_\lambda S^2$ respectively $S^2\times S^2\setminus\Delta$ generated by $\xi$, where $\exp:\mathfrak{so}(3)\to \text{SO}(3)$ denotes the exponential map. By Lemma \ref{lem5} these vector fields are related via
$$
\dd F_{(x,v)} \rho(\xi)^{D_\lambda S^2}_{(x,v)}=\diff F(\exp(t\xi)\cdot (x,v))=\diff \exp(t\xi)\cdot F(x,v)=\rho(\xi)^{S^2\times S^2}_{F(x,v)}.
$$
We denote $\Tilde{\omega}=\sigma\oplus R\sigma$. By construction $F$ satisfies
$\mu_{D_\lambda S^2}\circ F=\mu_{S^2\times S^2}$ thus 
\begin{align*}
    \omega_s(\rho^{D_\lambda S^2}(\xi),\cdot)&=-\dd(\langle\mu_{D_\lambda S^2},\xi\rangle)=-\dd(\langle\mu_{S^2\times S^2}\circ F,\xi\rangle)=-\dd(\langle\mu_{S^2\times S^2},\xi\rangle)\circ \dd F\\
    &=\Tilde{\omega}(\rho^{S^2\times S^2}(\xi),\dd F\ \cdot)=\Tilde \omega(\dd F\rho^{D_\lambda S^2}(\xi),\dd F\ \cdot)=F^*\tilde\omega(\rho^{D_\lambda S^2}(\xi),\cdot).
\end{align*}
As symplectic forms are skew symmetric and outside the zero-section the vector fields $\rho^{D_\lambda S^2}(\xi)$ generated by $\mathfrak{so}(3)$ span a three dimensional subspace of $TD_\lambda S^2$ this already proves
$$
F^ *(\sigma\oplus R\sigma)=\omega_s.
$$
\begin{figure}[h]
	\centering
  \includegraphics[width=0.95\textwidth]{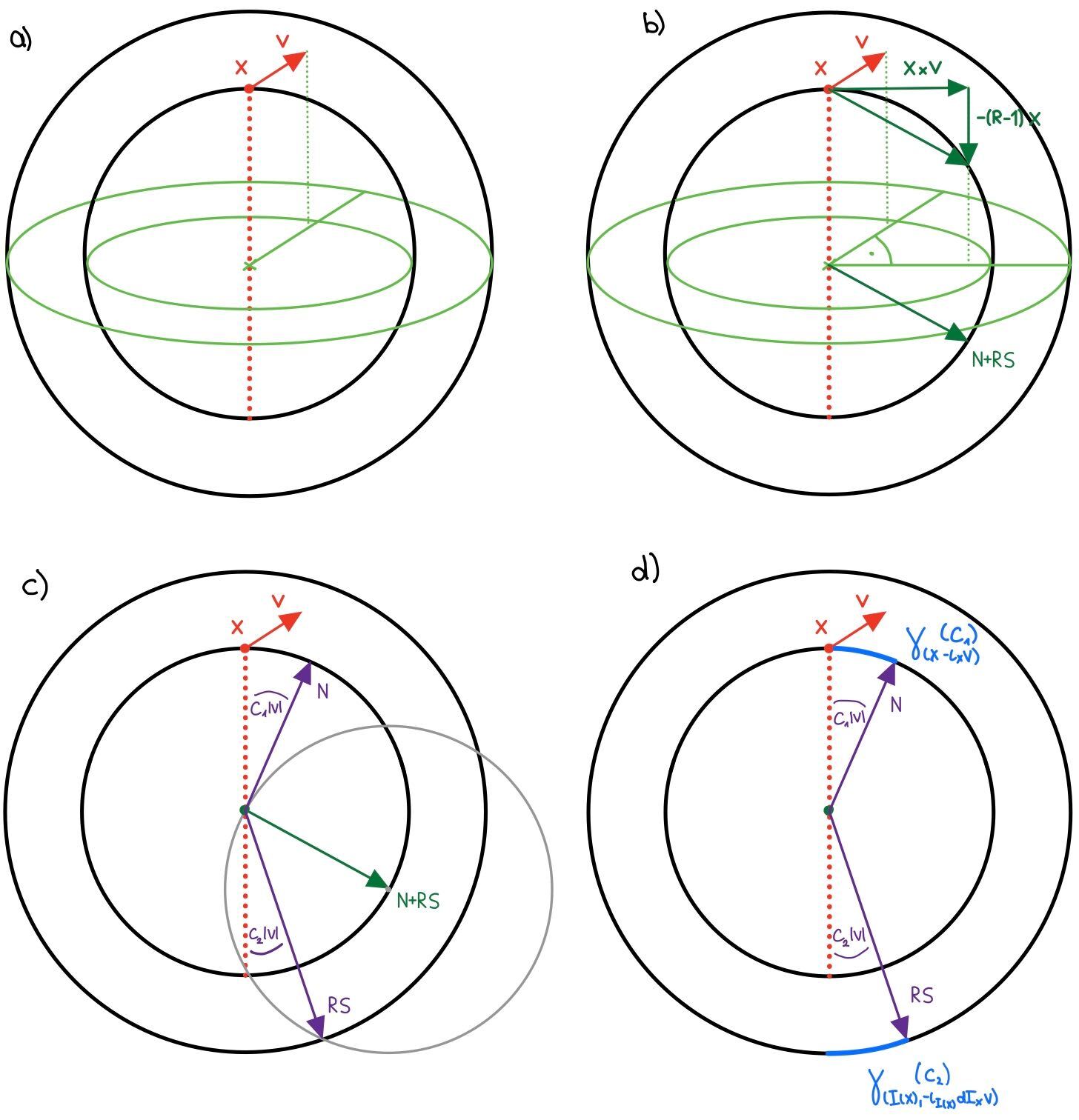}
	\caption{\textit{a) We take any $(x,v)\in TS^2$ and b) calculate the moment map
	$ \mu(x,v)=x\times v-(R-1)x. $
	We know that $\mu(x,v)\stackrel{!}{=}N+RS$. c) From now on we can reduce the picture to the two-dimensional intersection of the sphere with the plane spanned by $x$ and $x\times v$. To determine $RS$ draw a circle of radius $1$ around $N+RS$. It intersects the outer black circle (it has radius $R$) twice and the intersection that is further away from $x$ is $RS$. Then $N$ is determined by $N=N+RS-RS$. Observe that there is an intersection of the grey circle with the outer black circle only if $\vert v\vert\leq 2\sqrt{R}$ thus $\lambda=2\sqrt{R}$. d) We have geometrically determined the angles $c_1\vert v\vert$ and $c_2\vert v\vert$.}}
	\label{fig:8}
\end{figure}
\begin{figure}
	\centering
  \includegraphics[width=1\textwidth]{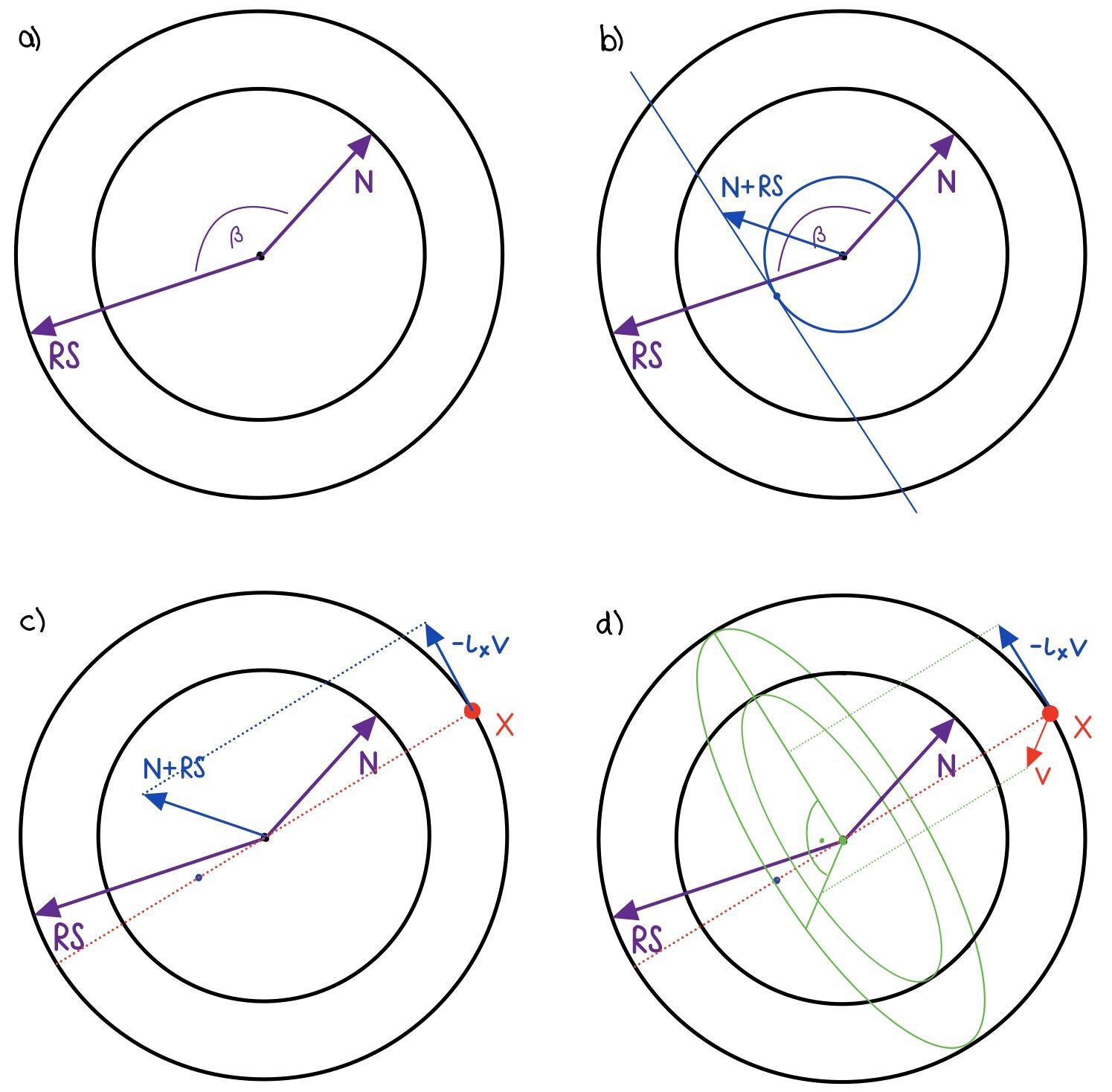}
	\caption{\textit{a) If $S\neq -N$, then $(N,S)\in S^2\times S^2\setminus\Delta$ determine a unique plane they both lie in. b) We can then draw a straight line from $N+RS$ tangential to the circle with radius $R-1$ (we choose the intersection closer to $S$). c) Then $x$ is determined as the intersection of the red line through the tangency and the origin, with the larger black circle. We can then determine $-\iota_x v=x\times v=N+RS+(R-1)x$ and d) deduce $v$. In this case $v$ points out of the plane. This map extends smoothly to $(N,-N)\in S^2\times S^2\setminus \Delta$ and maps it to $(N,0)\in D_\lambda S^2$.}}
	\label{fig:10}
\end{figure}
\bibliography{ref}{}

\begin{thebibliography}{10}

\bibitem{Bd14}
Gabriele Benedetti.
\newblock The contact property for magnetic flows on surfaces.
\newblock {\em PhD Thesis, University of Cambridge}, (see
  https://doi.org/10.17863/CAM.16235, or arXiv:1805.04916), 2014.

\bibitem{Bd16}
Gabriele Benedetti.
\newblock Magnetic {K}atok examples on the two-sphere.
\newblock {\em Bulletin of the London Mathematical Society}, 48(5):855--865,
  2016.

\bibitem{BR19}
Gabriele Benedetti and Alexander~F. Ritter.
\newblock Invariance of symplectic cohomology and twisted cotangent bundles
  over surfaces.
\newblock {\em Internat. J. Math.}, 31(9):2050070, 56, 2020.

\bibitem{BZ15}
Gabriele Benedetti and Kai Zehmisch.
\newblock On the existence of periodic orbits for magnetic systems on the
  two--sphere.
\newblock {\em Journal of Modern Dynamics}, 9(01):141--146, 2015.

\bibitem{FHV89}
A.~Floer, H.~Hofer, and C.~Viterbo.
\newblock The {W}einstein conjecture in {$P\times {\bf C}^l$}.
\newblock {\em Math. Z.}, 203(3):469--482, 1990.

\bibitem{Gb96}
Victor~L. Ginzburg.
\newblock On closed trajectories of a charge in a magnetic field. {A}n
  application of symplectic geometry.
\newblock {\em Contact and Symplectic Geometry}, pages 131--148, 1996.

\bibitem{Gb04}
Viktor~L. Ginzburg.
\newblock The {W}einstein conjecture and theorems of nearby and almost
  existence.
\newblock In {\em The breadth of symplectic and {P}oisson geometry}, volume 232
  of {\em Progr. Math.}, pages 139--172. Birkh{\"a}user Boston, Boston, MA,
  2005.

\bibitem{VG2003}
Viktor~L Ginzburg and Ba{\c s}ak~Z G{\"u}rel.
\newblock Relative {H}ofer-{Z}ehnder capacity and periodic orbits in twisted
  cotangent bundles.
\newblock {\em Duke Mathematical Journal}, 123(1):1--47, 2004.

\bibitem{Gr85}
M.~L. Gromov.
\newblock Pseudo holomorphic curves in symplectic manifolds.
\newblock {\em Inventiones Mathematicae}, 82:307--347, 1985.

\bibitem{HV92}
H.~Hofer and C.~Viterbo.
\newblock The {W}einstein conjecture in the presence of holomorphic spheres.
\newblock {\em Communications on Pure and Applied Mathematics}, 45(5):583--622,
  1992.

\bibitem{HZ94}
Helmut Hofer and Eduard Zehnder.
\newblock {\em Symplectic invariants and Hamiltonian dynamics}.
\newblock Birkh{\"a}user, 2011.

\bibitem{ler}
Eugene Lerman.
\newblock Symplectic cuts.
\newblock {\em Math. Res. Lett.}, 2(3):247--258, 1995.

\bibitem{LT00}
GANG Liu and GANG Tian.
\newblock Weinstein conjecture and {GW}-invariants.
\newblock {\em Communications in Contemporary Mathematics}, 2(04):405--459,
  2000.

\bibitem{Lu06}
Guangcun Lu.
\newblock Gromov-{W}itten invariants and pseudo symplectic capacities.
\newblock {\em Israel Journal of Mathematics}, 156(1):1--63, 2006.

\bibitem{Mac03}
Leonardo Macarini.
\newblock Hofer-{Z}ehnder capacity and {H}amiltonian circle actions.
\newblock {\em Communications in Contemporary Mathematics}, 06(06), 2002.

\bibitem{Wbr06}
Joa Weber.
\newblock Noncontractible periodic orbits in cotangent bundles and {F}loer
  homology.
\newblock {\em Duke Mathematical Journal}, 133(3):527--566, 2006.

\bibitem{Wdl18}
Chris Wendl.
\newblock {\em Holomorphic curves in low dimensions}.
\newblock Springer, 2018.

\end{thebibliography}
\bibliographystyle{plain}

\end{document}